\documentclass[12pt,oneside,english]{amsart}
\usepackage[T1]{fontenc}
\usepackage[latin9]{inputenc}
\usepackage{geometry}
\usepackage{amsthm}
\usepackage{amstext}
\usepackage{amssymb}

\makeatletter
\numberwithin{equation}{section}
\numberwithin{figure}{section}
\theoremstyle{plain}
\newtheorem{thm}{\protect\theoremname}
  \theoremstyle{plain}
  \newtheorem{lem}[thm]{\protect\lemmaname}
  \theoremstyle{remark}
  \newtheorem*{rem*}{\protect\remarkname}
  \theoremstyle{plain}
  \newtheorem{cor}[thm]{\protect\corollaryname}

\makeatother

\usepackage{babel}
  \providecommand{\corollaryname}{Corollary}
  \providecommand{\lemmaname}{Lemma}
  \providecommand{\remarkname}{Remark}
\providecommand{\theoremname}{Theorem}

\begin{document}

\title{{\normalsize{}Axioms for the Real Numbers: A Constructive Approach}}

\author{Jean S. Joseph}
\begin{abstract}
{\normalsize{}We present axioms for the real numbers by omitting the
field axioms and then derive the field properties of the real numbers.
We prove all our theorems constructively.}{\normalsize \par}

{\normalsize{}\tableofcontents{}}{\normalsize \par}
\end{abstract}

\maketitle

$\vphantom{}$

$\vphantom{}$

\textit{Email}: jsjean00@gmail.com

\section{Introduction}

We axiomatically define the real numbers as the completion of the
rational numbers. We keep the usual strict order and field structure
on the rational numbers, and we embed the rational numbers in a certain
way into a complete set; we will later say what that certain way is
and what we mean by `complete'. That complete set with the rational
numbers is what we call the completion of the rational numbers. We
do not need the field axioms to say what the real numbers are because
the ring properties of the rational numbers can be extended to the
completion (Theorem \ref{thm: 18}) and the invertibility of the nonzero
real numbers can be proved (Theorem \ref{thm: 20}). 

Much of the theory of the completion of an ordered set has been developed
in \cite{key-5}; we will use here only what we need from \cite{key-5}. 

We prove our theorems constructively, meaning we abstain from using
the law of excluded middle. More about constructive mathematics can
be found in \cite{key-1,key-2-1,key-4} or on the Web. Some of the
theorems below have been proved in \cite{key-5}, so their proofs
are omitted.

\section{Completion of an Ordered Set}

In \cite{key-5}, an \textit{ordered set} is defined as a set $X$
with a binary relation $<$ such that, for all $x,y,z\in X$, 
\begin{itemize}
\item $x<y$ implies $y<x$ is false; (Asymmetry)
\item $x<y$ implies $x<z$ or $z<y$; (Cotransitivity)
\item $x<y$ is false and $y<x$ is false imply $x=y$; (Negative Antisymmetry).
\end{itemize}
We write $x\leq y$ for $y<x$ is false. Also, in \cite{key-5}, a
subset $S$ of an ordered set $X$ is \textit{almost dense} in $X$
if $x<y$ in $X$ implies $x\leq s<s'\leq y$ for some $s,s'\in S$,
and $S$ is \textit{bicofinal} in $X$ if, for each $x\in X$, $s\leq x\leq s'$
for some $s,s'\in S$. We say $S$ is \textit{finitely enumerable}
if $S$ is empty or there is a positive integer $n$ and a function
from $\left\{ 1,\ldots,n\right\} $ onto $S$. A subset $S$ of $X$
is \textit{upper order located} if, $x<y$ in $X$ implies either
$x<s$ for some $s\in S$ or $u<y$ for some upper bound $u$ of $S$.
An ordered set $X$ is \textit{complete} if each nonempty, bounded
above, and upper order located subset of $X$ has a supremum in $X$.
We call a nonempty, bounded above, and upper order located subset
of $X$ a \textit{supable} subset of $X$. 

In \cite{key-5}, a \textit{completion} of an ordered set $X$ is
an ordered set $Y$ together with an embedding $f$ of $X$ into $Y$
such that 
\begin{itemize}
\item $Y$ is complete; 
\item $f\left(X\right)$ is almost dense in $Y$, and 
\item $f\left(X\right)$ is bicofinal in $Y$.
\end{itemize}
An \textit{embedding} $f$ is a function such that $x<x'$ if and
only if $f\left(x\right)<f\left(x'\right)$. In \cite{key-5}, we
also construct a completion of an arbitrary ordered set and show that
a completion of an ordered set is unique, up to isomorphism. As mentioned
above, we keep the usual strict order and the field structure on the
rational numbers $\mathbb{Q}$, and we define the real numbers $\mathbb{R}$
as the completion of $\mathbb{Q}$.

\section{Field Structure on the Real Numbers}

In what follows, we will show under what conditions we can define
a unique addition, a unique additive inversion, and a unique multiplication
on $\mathbb{R}$ and then prove that $\mathbb{R}$ is a field.

\subsection{Archimedean ordered abelian groups}

In \cite{key-5}, we prove that the completion of an Archimedean ordered
abelian group is an Archimedean ordered abelian group. Let us give
some of the details. 

An ordered abelian group is \textit{Archimedean} if $x,y>0$ implies
there is a positive integer $n$ such that $x\leq ny$. An addition
on an ordered set is \textit{compatible with the order} if $x+z<y+z$
if and only if $x<y$, and $z+x<z+y$ if and only if $x<y$. If $X$
is an ordered set containing an Archimedean ordered abelian group
$A$ as an almost dense, bicofinal subset, we say an addition on $X$
is \textit{admissible} if it extends the addition on $A$ and is compatible
with the order on $X$. The following is proved in \cite{key-5} (Theorem
58):
\begin{thm}
Let $A$ be an Archimedean ordered abelian group, and let $X$ be
an ordered set containing $A$ as an almost dense, bicofinal subset.
Then any two admissible additions on $X$ are equal.
\end{thm}
To define an admissible addition on an ordered set $X$ containing
an Archimedean ordered abelian group $A$ as an almost dense, bicofinal
subset, we need this lemma, which is proved in \cite{key-5} (Lemmas
59 \& 60):
\begin{lem}
Let $A$ be an Archimedean ordered abelian group. \label{lem: 2}
\begin{enumerate}
\item A subset $S$ of $A$ is supable if and only if, for each $\epsilon>0$
in $A$ and for each positive integer $m$, there exist $a\in S$
and an upper bound $u$ of $S$ in $A$ such that $m\left(u-a\right)<\epsilon$. 
\item The sum of two supable subsets of $A$ is supable. 
\end{enumerate}
\end{lem}
For an Archimedean ordered abelian group $A$, an admissible addition
on the completion $\overline{A}$ is defined as follows: for $x,y\in\overline{A}$,
consider the supable subsets $L_{x}=\left\{ a\in A:a\leq x\right\} $
and $L_{y}=\left\{ a\in A:a\leq y\right\} $. Since $L_{x}+L_{y}$
is supable by Lemma \ref{lem: 2}(2), we define $x+y$ to be $\sup\left(L_{x}+L_{y}\right)$.

For the additive inversion, we define $-x$ to be $\inf-L_{x}$, for
each $x\in\overline{A}$. 
\begin{lem}
Let $A$ be an Archimedean ordered abelian group, and let $X$ be
an ordered set containing $A$ as an almost dense, bicofinal subset.
Let $x,y\in X$. If, for each $\epsilon>0$ in $A$, there are $u,v\in A$
such that $x,y\in\left[u,v\right]$ and $v-u<\epsilon$, then $x=y$.
\label{lem: 3} (\cite{key-5}, Lemma 56)\end{lem}
\begin{thm}
Let $A$ be an Archimedean ordered abelian group. Then $\overline{A}$
is an Archimedean ordered abelian group. \label{thm: 3} (\cite{key-5},
Theorem 62).
\end{thm}

\subsection{Commutative ordered monoids}

We call a multiplication on an ordered set $X$ with a distinguished
element $0$ \textit{preadmissible }if, for each $x,y,z$ in $X$,
the following conditions hold:
\begin{enumerate}
\item $0x=0=x0$;
\item if $x<y$ and $z>0$, then $xz<yz$;
\item if $x<y$ and $z<0$, then $xz>yz$.\end{enumerate}
\begin{lem}
Let $X$ be an ordered set and $x,y,x',y'\in X$: \label{lem: 5}
\begin{enumerate}
\item If $x<y$ implies $x'<y'$, and $x=y$ implies $x'=y'$, then $x\leq y$
implies $x'\leq y'$.
\item If $x<y$ implies $x'>y'$, and $x=y$ implies $x'=y'$, then $x\leq y$
implies $x'\geq y'$. 
\end{enumerate}
\end{lem}
\begin{proof}
(1) Suppose $x\leq y$ and $y'<x'$. If $x<y$, then $x'<y'$, which
is false by asymmetry, so $y\leq x$. By negative antisymmetry, $x=y$,
so $x'=y'$, implying $x'<x'$, which is false by asymmetry.

(2) The proof goes as in (1).\end{proof}
\begin{thm}
Let $X$ be an ordered set with a distinguished element $0$ and a
preadmissible multiplication. For each $x,y,z$ in $X$, \label{thm: 40}
\begin{enumerate}
\item if $x\leq y$ and $z>0$, then $xz\leq yz$;
\item if $x\leq y$ and $z<0$, then $xz\geq yz$.
\end{enumerate}
\end{thm}
\begin{proof}
(1) Suppose $z>0$. If $x<y$, then $xz<yz$, because of condition
$2$ for a preadmissible multiplication. Also, if $x=y$, then $xz=yz$
because multiplication, as a function on $X\times X$, is well defined.
Hence, $x\leq y$ implies $xz\leq yz$ by Lemma \ref{lem: 5}(1).

(2) Suppose $z<0$. If $x<y$, then $xz>yz$ because of condition
$3$ for a preadmissible multiplication. Also, if $x=y$, then $xz=yz$.
Hence, $x\leq y$ implies $xz\geq yz$, by Lemma\ref{lem: 5}(2).\end{proof}
\begin{rem*}
In the proof of Theorem \ref{thm: 40}(1), we have shown 
\[
z>0\Rightarrow\left[\left(\left(x<y\Rightarrow xz<yz\right)\:\textrm{and}\:\left(x=y\Rightarrow xz=yz\right)\right)\Rightarrow\left(x\leq y\Rightarrow xz\leq yz\right)\right],
\]

which is equivalent to condition $2$ for a preadmissible multiplication
implying Theorem \ref{thm: 40}(1). That move is a special form of
the general reasoning law: 
\[
C\Rightarrow\left[\left(A\Rightarrow B\right)\Rightarrow\left(A'\Rightarrow B'\right)\right]\:\textrm{is equivalent to}\:\left((C\:\textrm{and}\: A)\Rightarrow B\right)\Rightarrow\left((C\:\textrm{and}\: A')\Rightarrow B'\right).
\]
\end{rem*}
\begin{thm}
Let $X$ be an ordered set with a distinguished element $0$ and a
preadmissible multiplication. Let $a,b,x,y$ be in $X$ with $a\leq x\leq b$.
Then\label{thm: 2-1}
\begin{enumerate}
\item if $l$ is a lower bound of $\left\{ ay,by\right\} $ and $u$ is
an upper bound of $\left\{ ay,by\right\} $, then $l\leq xy\leq u$;
\item if $l'$ is a lower bound of $\left\{ ya,yb\right\} $ and $u'$ is
an upper bound of $\left\{ ya,yb\right\} $, then $l'\leq yx\leq u'$.
\end{enumerate}
\end{thm}
\begin{proof}
(1) If $xy<l$, then $xy<ay$ and $xy<by$. If $y>0$, then $ay\leq xy\leq by$
by Theorem \ref{thm: 40}(1), which is false since $xy<ay$, so $y\leq0$.
If $y<0$, then $ay\geq xy\geq by$ by Theorem \ref{thm: 40}(2),
which is false since $xy<by$, so $y\geq0$. Hence, $y=0$, by negative
antisymmetry. But if $y=0$, then $x0<a0$, giving $0<0$, which is
false by asymmetry. Therefore, $l\leq xy$. Now, if $u<xy$, then
$ay<xy$ and $by<xy$. If $y>0$, then $ay\leq xy\leq by$, which
is false, so $y\leq0$. If $y<0$, then $ay\geq xy\geq by$, which
is false, so $y\geq0$. Hence, $y=0$. But if $y=0$, then $a0<x0$,
giving $0<0$, which is false. Therefore, $xy\leq u$.

(2) The proof goes as in (1).\end{proof}
\begin{thm}
Let $a\leq x\leq b$ and $a'\leq y\leq b'$ in $X$. If $l$ is a
lower bound of $\left\{ aa',ab',ba',bb'\right\} $ and $u$ is an
upper bound of $\left\{ aa',ab',ba',bb'\right\} $, then $l\leq xy\leq u$.
\label{thm: 41}\end{thm}
\begin{proof}
Since $u$ is an upper bound of $\left\{ aa',ab',ba',bb'\right\} $,
$u$ is also an upper bound of $\left\{ aa',ab'\right\} $ and an
upper bound of $\left\{ ba',bb'\right\} $. By Theorem \ref{thm: 2-1}(2),
$ay\leq u$ and $by\leq u$, so $u$ is an upper bound of $\left\{ ay,by\right\} $.
Therefore, $xy\leq u$, by Theorem \ref{thm: 2-1}(1). Similarly,
$l\leq xy$.
\end{proof}
By a \textit{commutative ordered monoid} $M$ with a distinguished
element $0$, we mean a monoid such that $xy=yx$, and $0<x$ and
$0<y$ imply $0<xy$. We call $M$ \textit{locally bounded} if every
finitely enumerable subset of $M$ has a minimum element and a maximum
element. 

For an ordered set $X$ containing a locally bounded commutative ordered
monoid $M$ with a distinguished element $0$, as an almost dense,
bicofinal subset, and satisfying $0\leq1$, a multiplication on $X$
is \textit{admissible} if 
\begin{enumerate}
\item the multiplication on $X$ is preadmissible;
\item the multiplication on $X$ extends the multiplication on $M$;
\item for $c,d\in M$, $c\leq xy\leq d$ implies there are $a,b,a',b'\in M$
such that $a\leq x\leq b$, $a'\leq y\leq b'$, and $c\leq\min\left\{ aa',ab',ba',bb'\right\} \leq xy\leq\max\left\{ aa',ab',ba',bb'\right\} \leq d$.
\end{enumerate}
In what remains, we will assume that $0\leq1$ in $M$.
\begin{lem}
Let $X$ be an ordered set containing a locally bounded commutative
ordered monoid $M$ with a distinguished element $0$, as an almost
dense, bicofinal subset. Let $X$ have a preadmissible multiplication.
Let $x,y$ be in $X$ and let $a,b,a',b'$ be in $M$. If $a\leq x\leq b$
and $a'\leq y\leq b'$, then $\min\left\{ aa',ab',ba',bb'\right\} \leq xy\leq\max\left\{ aa',ab',ba',bb'\right\} $.
\label{lem: 58}\end{lem}
\begin{proof}
Since $\min\left\{ aa',ab',ba',bb'\right\} $ is a lower bound of
$\left\{ aa',ab',ba',bb'\right\} $ and $\max\left\{ aa',ab',ba',bb'\right\} $
is an upper bound of $\left\{ aa',ab',ba',bb'\right\} $, 
\[
\min\left\{ aa',ab',ba',bb'\right\} \leq xy\leq\max\left\{ aa',ab',ba',bb'\right\} ,
\]

by Theorem \ref{thm: 41}.\end{proof}
\begin{lem}
Let $X$ be an ordered set containing an almost dense, bicofinal subset
$M$. Let $x,y\in X$. If, for all $a,b\in M$, $a\leq x\leq b$ if
and only if $a\leq y\leq b$, then $x=y$. \label{lem: 9}\end{lem}
\begin{proof}
If $x<y$, then $a\leq x\leq c<d\leq y\leq b$ for some $a,c,d,b\in M$,
since $M$ is almost dense and bicofinal in $X$, so $c<x$ and $y<d$,
which is impossible, by asymmetry. Therefore, $y\leq x$. Symmetrically,
$x\leq y$. Hence, $x=y$, by negative antisymmetry.\end{proof}
\begin{thm}
Let $X$ be an ordered set containing a locally bounded commutative
ordered monoid $M$ with a distinguished element $0$, as an almost
dense, bicofinal subset. Any two admissible multiplications on $X$
are equal.\end{thm}
\begin{proof}
For two admissible multiplications $\hat{\cdot}$ and $\tilde{\cdot}$
on $X$, suppose for all $c,d\in M$, $c\leq x\hat{\cdot}y\leq d$.
Then there are $a,b,a',b'\in M$ such that $a\leq x\leq b$, $a'\leq y\leq b'$,
and $c\leq\min\left\{ aa',ab',ba',bb'\right\} \leq x\hat{\cdot}y\leq\max\left\{ aa',ab',ba',bb'\right\} \leq d$.
Therefore, $c\leq\min\left\{ aa',ab',ba',bb'\right\} \leq x\tilde{\cdot}y\leq\max\left\{ aa',ab',ba',bb'\right\} \leq d$
by Lemma \ref{lem: 58}, so $c\leq x\tilde{\cdot}y\leq d$. Similarly,
$c\leq x\tilde{\cdot}y\leq d$ implies $c\leq x\hat{\cdot}y\leq d$.
Hence, $x\hat{\cdot}y=x\tilde{\cdot}y$, by Lemma \ref{lem: 9}.\end{proof}
\begin{lem}
Let $X$ be an ordered set containing a locally bounded commutative
ordered monoid $M$ with a distinguished element $0$, as an almost
dense, bicofinal subset. Let $x,y\in X$. Suppose also that the multiplication
on $M$ be preadmissible. For fixed $r_{0},s_{0},r_{0}',s_{0}'\in M$
and for all $a,b,a',b'\in M$, if $r_{0},a\leq x\leq s_{0},b$ and
$r_{0}',a'\leq y\leq s_{0}',b'$, then $\min\left\{ aa',ab',ba',bb'\right\} \leq\max\left\{ r_{0}r_{0}',r_{0}s_{0}',s_{0}r_{0}',s_{0}s_{0}'\right\} $.
\label{lem: 11}\end{lem}
\begin{proof}
Since $M$ is locally bounded and almost dense in $X$, there are
$m,m'\in M$ such that $r_{0},a\leq m\leq s_{0},b$ and $r_{0}',a'\leq m'\leq s_{0}',b'$.
Therefore, 
\[
\min\left\{ r_{0}r_{0}',r_{0}s_{0}',s_{0}r_{0}',s_{0}s_{0}'\right\} ,\min\left\{ aa',ab',ba',bb'\right\} \leq mm'\leq\max\left\{ r_{0}r_{0}',r_{0}s_{0}',s_{0}r_{0}',s_{0}s_{0}'\right\} ,\max\left\{ aa',ab',ba',bb'\right\} 
\]
 by Lemma \ref{lem: 58}.\end{proof}
\begin{thm}
Let $X$ be an ordered set containing a locally bounded commutative
ordered monoid $M$ with a distinguished element $0$, as an almost
dense, bicofinal subset. In addition, assume that the multiplication
on $M$ be preadmissible and satisfy the following ({*}): for $x,y\in X$,
if $c<d$ in $M$, then there are $a,b,a',b'\in M$ such that $a\leq x\leq b$;
$a'\leq y\leq b'$; and either $c<\min\left\{ aa',ab',ba',bb'\right\} $
or $\max\left\{ aa',ab',ba',bb'\right\} <d$. Let $x,y\in X$; then
the subset $P_{x,y}=\left\{ \min\left\{ aa',ab',ba',bb'\right\} :a,b,a',b'\in M,\textrm{ }a\leq x\leq b\textrm{ and }a'\leq y\leq b'\right\} $
is supable. \label{thm: 12}\end{thm}
\begin{proof}
The subset $P_{x,y}$ is nonempty because $M$ is bicofinal in $X$
and $M$ is locally bounded, and $P_{x,y}$ is bounded above by Lemma
\ref{lem: 11}. Upper order locatedness of $P_{x,y}$ follows from
almost density of $M$ in $X$ and from ({*}). 
\end{proof}
Let $M$ be a locally bounded commutative ordered monoid with a distinguished
element $0$, and let the multiplication on $M$ be continuous%
\footnote{More precisely, the multiplication on $M$ is continuous in the product
topology. In \cite{key-5} (p. 18), the order topology on an ordered
set is introduced. %
}, be preadmissible, and satisfy the condition ({*}) in Theorem \ref{thm: 12}.
For $x,y\in\overline{M}$, let the multiplication on $\overline{M}$
be $xy=\sup P_{x,y}$, with $P_{x,y}$ being defined in Theorem \ref{thm: 12}.
Under these hypotheses, we will show the following theorems about
this multiplication on $\overline{M}$.
\begin{thm}
For each $x\in\overline{M}$, $x0=0=0x$.\end{thm}
\begin{proof}
If $0<x0$, then $0<\min\left\{ aa',ab',ba',bb'\right\} $ for some
$a,b,a',b'\in M$ with $a\leq x\leq b$ and $a'\leq0\leq b'$ because
$x0=\sup P_{x,0}$, by definition. Then
\[
\begin{array}{c}
0<aa',\\
0<ab',\\
0<ba',\\
0<bb'.
\end{array}
\]
Since $0<aa'$, it is false that $a=0$ and it is false that $a>0$.
Since $0<ab'$, it is false that $a=0$ and it is false that $a<0$.
Hence, $a=0$ by negative antisymmetry, which is impossible since
the multiplication on $M$ is admissible, $0<aa'$, and $0<ab'$.
Therefore, $x0\leq0$. 

If $x0<0$, then $x0\leq r<s\leq0$ for some $r,s\in M$, since $M$
is almost dense in $\overline{M}$. By condition ({*}) in Theorem
\ref{thm: 12}, there are $a,b,a',b'\in M$ such that $a\leq x\leq b$;
$a'\leq0\leq b'$; and either $r<\min\left\{ aa',ab',ba',bb'\right\} $
or $\max\left\{ aa',ab',ba',bb'\right\} <s$. If $r<\min\left\{ aa',ab',ba',bb'\right\} $,
then $\min\left\{ aa',ab',ba',bb'\right\} \leq x0<r<\min\left\{ aa',ab',ba',bb'\right\} $,
which is impossible by asymmetry. If $\max\left\{ aa',ab',ba',bb'\right\} <s$,
then 
\[
\begin{array}{c}
aa'<0,\\
ab'<0,\\
ba'<0,\\
bb'<0.
\end{array}
\]
 Since $aa'<0$, it is false that $a=0$ and it is false that $a<0$.
Since $ab'<0$, it is false that $a=0$ and it is false that $a>0$.
Hence, $a=0$ by negative antisymmetry, which is impossible. Therefore,
$x0\geq0$. 

Since $x0\leq0$ and $x0\geq0$, it follows $x0=0$, by negative antisymmetry.
Similarly, $0=0x$.\end{proof}
\begin{thm}
Let $x,y\in\overline{M}$. For $c,d\in M$, $c\leq xy\leq d$ implies
there are $a,b,a',b'\in M$ such that $a\leq x\leq b$, $a'\leq y\leq b'$,
and $c\leq\min\left\{ aa',ab',ba',bb'\right\} \leq xy\leq\max\left\{ aa',ab',ba',bb'\right\} \leq d$.
\label{thm: 14}\end{thm}
\begin{proof}
In \cite{key-5} (Theorem 32), we proved that a nonempty almost dense
subset of an ordered set is topologically dense in that ordered set,
so $M$ is topologically dense in $\overline{M}$. Since elements
of $M$ can be arbitrarily close to $x,y$ and to themselves, elements
of $M\times M$ can be arbitrarily close to $\left(x,y\right)$ and
to themselves in the product topology. Hence, take $a,b,a',b'\in M$
close enough to $x,y$ such that $a\leq x\leq b$ and $a'\leq y\leq b'$
and such that $aa',ab',ba',bb'$ are very close to each other by continuity
of multiplication on $M$. Therefore, $c\leq\min\left\{ aa',ab',ba',bb'\right\} \leq xy\leq\max\left\{ aa',ab',ba',bb'\right\} \leq d$;
note that $\min\left\{ aa',ab',ba',bb'\right\} \leq xy$ by the definition
of the multiplication on $\overline{M}$ and that $xy\leq\max\left\{ aa',ab',ba',bb'\right\} $
because $\max\left\{ aa',ab',ba',bb'\right\} $ is an upper bound
of $P_{x,y}$ by Lemma \ref{lem: 11} and $xy$ is the least upper
bound of $P_{x,y}$ in Theorem \ref{thm: 12}.\end{proof}
\begin{thm}
The multiplication on $\overline{M}$ extends the multiplication on
$M$.\end{thm}
\begin{proof}
Let $m,m'\in M$. Since $m\leq m\leq m$ and $m'\leq m'\leq m'$ and
since the multiplication on $M$ is preadmissible, $m\cdot_{M}m'\leq m\cdot_{\overline{M}}m'\leq m\cdot_{M}m'$,
by Lemma \ref{lem: 58}, so $m\cdot_{M}m'=m\cdot_{\overline{M}}m'$,
by negative antisymmetry.\end{proof}
\begin{thm}
Let $M$ be a locally bounded commutative ordered monoid with a distinguished
element $0$, and let the multiplication on $M$ be continuous, be
preadmissible, and satisfy the condition ({*}) in Theorem \ref{thm: 12}.
Then $\overline{M}$ is a commutative ordered monoid. \label{thm: 16}\end{thm}
\begin{proof}
Let $x,y,z\in\overline{M}$. For all $a,b\in M$, suppose $a\leq x\leq b$.
Since $1\in M$ and $0\leq1$, $a\leq x1\leq b$. Now suppose $a\leq x1\leq b$.
Then there are there are $c,d,a',b'\in M$ such that $c\leq x\leq d$,
$a'\leq1\leq b'$, and $a\leq\min\left\{ ca',cb',da',db'\right\} \leq x1\leq\max\left\{ ca',cb',da',db'\right\} \leq b$,
by Theorem \ref{thm: 14}. If $x<a$, then $e\leq x\leq m<m'\leq a$
for some $e,m,m'\in M$, so $e\leq x1\leq m$ by the definition of
the multiplication on $\overline{M}$. Hence, $x1<x1$, which is impossible,
so $a\leq x$. Similarly, $x\leq b$. Therefore, $x1=1$, by Lemma
\ref{lem: 9}. 

Associativity and commutativity of multiplication on $\overline{M}$
follows from associativity and commutativity on $M$ , Theorem \ref{thm: 14},
and Lemma \ref{lem: 9}.

Now suppose $0<x,y$. Then $0<a\leq x\leq b$ and $0<a'\leq y\leq b'$
for some $a,b,a',b'\in M$, since $M$ is almost dense and bicofinal
in $\overline{M}$. Since $M$ is an ordered monoid, $0<aa',ab',ba',bb'$,
so $0<\min\left\{ aa',ab',ba',bb'\right\} \leq xy$. Hence, $0<xy$.
\end{proof}
By a \textit{(locally bounded) commutative ordered ring} $X$, we
mean an ordered set with an addition $+$ and multiplication $\cdot$
such that $\left(X,+\right)$ is an ordered abelian group, $\left(X,\cdot\right)$
is a (locally bounded) commutative ordered monoid, and the distributive
law holds in $X$. We say $X$ is Archimedean if $\left(X,+\right)$
is Archimedean. Note that the multiplication on a commutative ordered
ring is preadmissible.
\begin{thm}
The real numbers $\mathbb{R}=\overline{\mathbb{Q}}$ are an Archimedean,
commutative ordered ring. \label{thm: 18}\end{thm}
\begin{proof}
By Theorem \ref{thm: 3}, $\left(\mathbb{R},+\right)$ is an Archimedean
ordered abelian group. By Theorem \ref{thm: 16}, $\left(\mathbb{R},\cdot\right)$
is a commutative ordered monoid. To prove the distributive law $x\left(y+z\right)=xy+xz$,
note that addition and multiplication on $\mathbb{R}$ are continuous
functions, so the functions on $\mathbb{R}^{3}$ defined by $L\left(x,y,z\right)=x\left(y+z\right)$
and by $R\left(x,y,z\right)=xy+xz$ are continuous. Since $\mathbb{Q}^{3}$
is topologically dense in $\mathbb{R}^{3}$ (because $\mathbb{Q}$
is topologically dense in $\mathbb{R}$) and since $L=R$ when restricted
to $\mathbb{Q}^{3}$, it follows $L=R$ on $\mathbb{R}$. Hence, the
distributive law holds in $\mathbb{R}$.
\end{proof}

\subsection{Invertibility of nonzero elements}

Any ordered set admits a tight apartness defined by $x\neq y$ if
$x<y$ or $y<x$. A tight apartness is a positive notion for the negative
notion of difference, which is two elements are ``different'' if
they are ``not equal''. A tight apartness is discussed in \cite{key-6-1}.
In an ordered set with a distinguished element $0$, we say an element
$x$ is \textit{nonzero} if $x\neq0$. An element $x$ in a monoid
$M$ is \textit{invertible} if $xy=1$ for some $y\in M$.
\begin{lem}
Let $x$ be a nonzero real number and $D_{x}=\left\{ 1/b\in\mathbb{Q}:0<x\leq b\:\textrm{or}\: x\leq b<0\right\} $.
Then $D_{x}$ is supable.\end{lem}
\begin{proof}
If $0<x$, let $\epsilon$ be a positive rational number and $r$
a rational number such that $0<r<x$. Then there are rational numbers
$a$ and $b$ such that $r<a\leq x\leq b$ and $b-a<r^{2}\epsilon$.
Also $r^{2}<ab$, $1/a$ is an upper bound of $D_{x}$, and 
\[
\left(1/b\right)-\left(1/a\right)=\left(a-b\right)/ba\leq\left(b-a\right)/ba<\epsilon.
\]

If $x<0$, let $r$ be a rational number such that $x<r<0$. Then
there are rational numbers $a$ and $b$ such that $a\leq x\leq b<r$
and $b-a<r^{2}\epsilon$. Then 
\[
\left(1/b\right)-\left(1/a\right)=\left(a-b\right)/ba<\epsilon.
\]

Therefore, $D_{x}$ is supable by Lemma \ref{lem: 2}(1).\end{proof}
\begin{thm}
A real number is nonzero if and only if it is invertible. \label{thm: 20}\end{thm}
\begin{proof}
Let $x$ be a nonzero real number. If $0<x$, let $\epsilon$ be a
positive rational number, and let $r$ and $s$ be rational numbers
such that $0<r<x<s$. Then there are rational numbers $a$ and $b$
such that $r<a\leq x\leq b<s$ and $b-a<r^{2}\epsilon/2s$, so 
\[
1/b\leq\sup D_{x}\leq1/a,
\]

by the definition of $D_{x}$. Since $\min\left\{ a/b,1,b/a\right\} =a/b$
and $\max\left\{ a/b,1,b/a\right\} =b/a$, 
\[
a/b\leq x\sup D_{x}\leq b/a
\]

by Lemma \ref{lem: 11} and the definition of multiplication. Also
$a/b\leq1\leq b/a$ because $0<a\leq b$, and 
\[
\left(b/a\right)-\left(a/b\right)=\left(b^{2}-a^{2}\right)/ab=\left(\left(b+a\right)/ab\right)\left(b-a\right)<\epsilon.
\]

Hence, $x\sup D_{x}=1$, by Lemma \ref{lem: 3}. If $x<0$, let $r'$
and $s'$ be rational numbers such that $r'<x<s'<0$. Then 

\[
r'<a\leq x\leq b<s'<0
\]
and 
\[
b-a<(s'^{2}/-2r')\epsilon
\]

for some rational numbers $a$ and $b$, so $1/b\leq\sup D_{x}\leq1/a$.
Thus, 
\[
b/a\leq x\sup D_{x}\leq a/b
\]

with $b/a\leq1\leq a/b$ and $\left(a/b\right)-\left(b/a\right)<\epsilon$.
Therefore, 
\[
x\sup D_{x}=1.
\]

Conversely, let $x$ be an invertible real number, so $xy=1$ for
some $y\in\mathbb{R}$. Then there is a positive integer $N$ such
that $-N<y<N$. Also, by cotransitivity, either $-1/N<x$ or $x<0$
, and either $0<x$ or $x<1/N$. It is impossible that $x<0$ and
$0<x$, by irreflexivity, nor is it possible that $-1/N<x<1/N$; if
it were, $-1<xy<1$. Therefore, $x<0$ or $0<x$, so $x\neq0$.
\end{proof}
A \textit{Heyting field} is a field with a tight apartness. Heyting
fields are discussed in \cite{key-6-1}.
\begin{cor}
The real numbers are a Heyting field. \label{cor: 20}\end{cor}
\begin{proof}
By Theorems \ref{thm: 18} and \ref{thm: 20}.\end{proof}

\end{document}